\documentclass[10pt,reqno]{amsart}
\usepackage{amsmath}
\usepackage{amsthm}
\usepackage{amssymb}
\usepackage{graphicx}
\usepackage{amsfonts}
\usepackage{latexsym}
\usepackage{setspace}

\newcommand{\D}{ \mathbb{D}}
\newcommand{\dD}{ \partial\mathbb{D}}

\newcommand{\norm}[1]{\| #1 \|}
\newcommand{\inner}[1]{\langle #1 \rangle}

\renewcommand{\phi}{\varphi}

\newtheorem{Corollary}{Corollary}
\newtheorem{Theorem}{Theorem}

\newtheorem{Lemma}{Lemma}
\theoremstyle{definition}

\newtheorem*{Question}{Question}

\allowdisplaybreaks

\begin{document}
\bibliographystyle{amsplain}

    \title{The Norm of a Truncated Toeplitz Operator}

    \author{Stephan Ramon Garcia}
    \address{   Department of Mathematics\\
            Pomona College\\
            Claremont, California\\
            91711 \\ USA}
    \email{Stephan.Garcia@pomona.edu}
    \urladdr{http://pages.pomona.edu/\textasciitilde sg064747}

    \author{William T. Ross}
\address{   Department of Mathematics and Computer Science\\
            University of Richmond\\
            Richmond, Virginia\\
            23173 \\ USA}
    \email{wross@richmond.edu}
    \urladdr{http://facultystaff.richmond.edu/~wross}


    \thanks{First author partially supported by National Science Foundation Grant DMS-0638789.}

    \begin{abstract}
	We prove several lower bounds for the norm of a truncated Toeplitz operator
	and obtain a curious relationship between the $H^2$ and $H^{\infty}$ norms of functions
	in model spaces.
    \end{abstract}

\maketitle

\section{Introduction}

In this paper, we continue the discussion initiated in \cite{NLEPHS} concerning the norm of a truncated Toeplitz operator.
In the following, let $H^2$ denote the classical Hardy space of the open unit disk $\D$ and $K_{\Theta} := H^2 \cap (\Theta H^2)^{\perp}$, where $\Theta$ is an inner function, denote one of the so-called Jordan model spaces \cite{CR, DSS, Nikolski}.  If $H^{\infty}$ is the set of all bounded analytic  functions on $\D$, the space
$K_{\Theta}^{\infty}:=H^{\infty} \cap K_{\Theta}$ is norm dense in $K_{\Theta}$
(see \cite[p.~83]{CR} or \cite[Lem.~2.3]{Sarason}).
If $P_{\Theta}$ is the orthogonal projection from
$L^2 := L^2(\dD, \frac{|d \zeta|}{2 \pi})$ onto $K_{\Theta}$ and $\phi \in L^2$, then the operator
\begin{equation*}
	A_{\phi} f := P_{\Theta} (\phi f), \quad f \in K_{\Theta}^{\infty},
\end{equation*}
is densely defined on $K_{\Theta}$ and is called a \textit{truncated Toeplitz operator}.
Various aspects of these operators were studied in \cite{MR2440673, G, NLEPHS, Sarason, MR2418122}.

If $\|\cdot\|$ is the norm on $L^2$, we let
\begin{equation}\label{eq-NormDef}
\|A_{\phi}\| := \sup\{\|A_{\phi} f\|: f \in K_{\Theta}^{\infty}, \|f\| = 1\}
\end{equation}
and note that this quantity is finite if and only if $A_{\phi}$ extends to a bounded operator on $K_{\Theta}$.
 When $\phi \in L^{\infty}$, the set of bounded measurable functions on $\dD$,  we have the basic estimates
\begin{equation*}
0 \leq \|A_{\phi}\| \leq \|\phi\|_{\infty}.
\end{equation*}
However, it is known that equality can hold, in nontrivial ways, in either of the inequalities above and hence
finding the norm of a truncated Toeplitz operator can be difficult.
Furthermore, it turns out that there are many unbounded symbols $\phi \in L^2$ which yield
bounded operators $A_{\phi}$.  Unlike the situation for classical Toeplitz operators
on $H^2$, for a given $\phi \in L^2$, there many $\psi \in L^2$ for which $A_{\phi} = A_{\psi}$ \cite[Thm.~3.1]{Sarason}.

For a given symbol $\phi \in L^2$ and inner function $\Theta$, lower bounds
on the quantity \eqref{eq-NormDef} are useful in answering the following nontrivial questions:
\begin{enumerate}\addtolength{\itemsep}{0.5\baselineskip}
	\item is $A_{\phi}$ unbounded?
	\item if $\phi \in L^{\infty}$, is $A_{\phi}$ norm-attaining (i.e., is $\norm{A_{\phi}} = \norm{\phi}_{\infty}$)?
\end{enumerate}

When $\Theta$ is a finite Blaschke product and $\phi \in H^{\infty}$, the paper \cite{NLEPHS} computes $\|A_{\phi}\|$ and gives necessary and sufficient conditions as to when $\|A_{\phi}\| = \|\phi\|_{\infty}$. The purpose of this short note
is to give a few lower bounds on $\|A_{\phi}\|$ for general inner functions $\Theta$ and general $\phi \in L^2$. Along the way,
we obtain a curious relationship (Corollary \ref{CorollaryOddRelationship})
between the $H^2$ and $H^{\infty}$ norms of functions in $K_{\Theta}^{\infty}$.

\section{Lower bounds derived from Poisson's formula}
	
	For $\phi \in L^2$, let
	\begin{equation} \label{PIF}
		(\mathfrak{P} \phi)(z)
		:= \int_{\dD} \frac{1 - |z|^2}{|\zeta - z|^2} \phi(\zeta) \frac{|d \zeta|}{2 \pi}, \quad z \in \D,
	\end{equation}
	be the standard Poisson extension of $\phi$ to $\D$.   For $\phi \in C(\dD)$, the continuous functions on $\dD$, recall that
	$\mathfrak{P} \phi$ solves the classical Dirichlet problem with boundary data $\phi$.
	Also note that
	\begin{equation*}
		k_{\lambda}(z) := \frac{1 - \overline{\Theta(\lambda)} \Theta(z)}{1 - \overline{\lambda} z},
		\quad \lambda, z \in \D,
	\end{equation*}
	is the reproducing kernel for $K_{\Theta}$ \cite{Sarason}.

	Our first result provides a general lower bound for $\norm{A_{\phi}}$ which yields a number of useful corollaries:

	\begin{Theorem} \label{PROP-MAIN}
		If $\phi \in L^2$, then
		\begin{equation} \label{eq-General-Lower}
			\sup_{\lambda \in \D} \frac{1 - |\lambda|^2}{1 - |\Theta(\lambda)|^2} \left|\int_{\dD} \phi(z)
			\left|\frac{\Theta(z) - \Theta(\lambda)}{z - \lambda}\right|^2 \frac{|dz|}{2 \pi}\right| \,\leq\, \norm{A_{\phi}}.
		\end{equation}
		In other words,
		\begin{equation*}
			\sup_{\lambda \in \D} \left| \int_{\dD} \phi(z) d\nu_{\lambda}(z) \right| \leq \norm{A_{\phi}}
		\end{equation*}
		where
		\begin{equation*}
			d\nu_{\lambda}(z) :=  \frac{1 - |\lambda|^2}{1 - |\Theta(\lambda)|^2}
			\left|\frac{\Theta(z) - \Theta(\lambda)}{z - \lambda}\right|^2 \frac{|dz|}{2 \pi}
		\end{equation*}
		is a family of probability measures on $\dD$ indexed by $\lambda \in \D$.
	\end{Theorem}

	\begin{proof}
		For $\lambda \in \D$ we have
		\begin{equation}\label{eq-NormKernel}
			\norm{k_{\lambda}} = \sqrt{\frac{1 - |\Theta(\lambda)|^2}{1 - |\lambda|^2}},
		\end{equation}
		from which it follows that
		\begin{align*}
			\norm{A_{\phi}}
			&\geq   \frac{1 - |\lambda|^2}{1 - |\Theta(\lambda)|^2} |\inner{ A_{\phi} k_{\lambda}, k_{\lambda} } | \\
			&=  \frac{1 - |\lambda|^2}{1 - |\Theta(\lambda)|^2} |\inner{ P_{\Theta} \phi k_{\lambda}, k_{\lambda} } | \\
			&=  \frac{1 - |\lambda|^2}{1 - |\Theta(\lambda)|^2} |\inner{ \phi k_{\lambda}, k_{\lambda} }| \\
			&=\frac{1 - |\lambda|^2}{1 - |\Theta(\lambda)|^2} \left|\int_{\dD} \phi(z)
			\left|\frac{\Theta(z) - \Theta(\lambda)}{z - \lambda}\right|^2 \frac{|dz|}{2 \pi}\right|.
		\end{align*}
		That the measures $d\nu_{\lambda}$ are indeed probability measures follows
		from \eqref{eq-NormKernel}.
	\end{proof}
%
%	We should remark that since the measures $d\nu_{\lambda}$ are nonnegative,
%	the inequality \eqref{eq-General-Lower} holds with either $\Re \phi$ or $\Im \phi$ in place of $\phi$.

	Now observe that if $\Theta(\lambda) = 0$, the argument in the supremum on the left hand side of \eqref{eq-General-Lower}
	becomes the absolute value of the expression in \eqref{PIF}.  This immediately yields the following corollary:

	\begin{Corollary}\label{TheoremHarmonic}
		If $\phi \in L^2$, then
		\begin{equation}\label{eq-HarmonicLower}
			 \sup_{ \lambda \in \Theta^{-1} ( \{0\} ) } | (\mathfrak{P}\phi)(\lambda) |
			 \,\leq \,\norm{A_{\phi}},
		\end{equation}
		where the supremum is to be regarded as $0$ if $\Theta^{-1} ( \{0 \}) = \varnothing$.
	\end{Corollary}

	Under the right circumstances, the preceding corollary can be used to prove that certain truncated
	Toeplitz operators are norm-attaining:

	\begin{Corollary}\label{CorollaryAccumulate}
		Let $\Theta$ be an inner function having zeros which accumulate at every point of $\dD$.
		If $\phi \in C(\dD)$  then $\norm{A_{\phi}} = \norm{\phi}_{\infty}$.
	\end{Corollary}

	\begin{proof}
		Let $\zeta \in \dD$ be such that
		 $|\phi(\zeta)| = \norm{\phi}_{\infty}$.  By hypothesis, there
		exists a sequence $\lambda_n$ of zeros of $\Theta$ which
		converge to $\zeta$.  By continuity, we conclude that
		\begin{equation*}
			\norm{\phi}_{\infty} = \lim_{n\to\infty} |(\mathfrak{P} \phi)(\lambda_n)|
			\leq \norm{A_{\phi}} \leq \norm{\phi}_{\infty}
		\end{equation*}
		whence $\norm{A_{\phi}} = \norm{\phi}_{\infty}$.
	\end{proof}
	
	The preceding corollary stands in contrast to the finite Blaschke product setting.
	Indeed, if $\Theta$ is a finite Blaschke product and $\phi \in H^{\infty}$, then it is known that
	$\|A_{\phi}\| = \|\phi\|_{\infty}$ if and only if $\phi$ is the scalar multiple of the
	inner factor of some function from $K_{\Theta}$ \cite[Thm.~2]{NLEPHS}.

	At the expense of wordiness, the hypothesis of Corollary \ref{CorollaryAccumulate} can be considerably
	weakened.  A cursory examination of the proof indicates that we only need $\zeta$
	to be a limit point of the zeros of $\Theta$, $\phi \in L^{\infty}$ to be continuous on an
	open arc containing $\zeta$, and $|\phi(\zeta)| = \norm{ \phi}_{\infty}$.
	
	\medskip
	
	Theorem \ref{PROP-MAIN} yields yet another lower bound for $\norm{A_{\phi}}$.
	Recall that an inner function $\Theta$ has a finite angular derivative at $\zeta \in \dD$
	if $\Theta$ has a non-tangential limit $\Theta(\zeta)$ of modulus one at $\zeta$
	and $\Theta'$ has a finite non-tangential limit $\Theta'(\zeta)$ at $\zeta$. This is
	equivalent to asserting that
	\begin{equation} \label{RKC}
		\frac{\Theta(z) - \Theta(\zeta)}{z - \zeta}
	\end{equation}
	has the non-tangential limit $\Theta'(\zeta)$ at $\zeta$.
	If $\Theta$ has a finite angular derivative at $\zeta$, then the function in \eqref{RKC} belongs to $H^2$ and
	\begin{equation*}
		\lim_{r \to 1^{-}} \int_{\dD} \left|\frac{\Theta(z) - \Theta(r \zeta)}{z - r \zeta}\right|^2 \frac{|d z|}{2 \pi}
		= \int_{\dD} \left|\frac{\Theta(z) - \Theta(\zeta)}{z - \zeta}\right|^2 \frac{|d z|}{2 \pi}.
	\end{equation*}
	Furthermore, the above is equal to
	\begin{equation*}
		\lim_{r \to 1^{-}} \frac{1 - |\Theta(r \zeta)|^2}{1 - r^2} = |\Theta'(\zeta)| > 0.
	\end{equation*}

	See \cite{CMR, MR1289670} for further details on angular derivatives.
	Theorem \ref{PROP-MAIN} along with the preceding discussion and
	Fatou's lemma yield the following lower estimate for $\norm{A_{\phi}}$.

	\begin{Corollary}
		For an inner function $\Theta$, let $D_{\Theta}$ be the set of $\zeta \in \dD$
		for which $\Theta$ has a finite angular derivative $\Theta'(\zeta)$ at $\zeta$.
		If $\phi \in L^{\infty}$ or if $\phi \in L^2$ with $\phi \geq 0$, then
		\begin{equation*}
			\sup_{\zeta \in D_{\Theta}} \frac{1}{|\Theta'(\zeta)|} \left|\int_{\dD} \phi(z)
			\left|\frac{\Theta(z) - \Theta(\zeta)}{z - \zeta}\right|^2 \frac{|dz|}{2 \pi}\right| \leq \norm{A_{\phi}}.
		\end{equation*}
			In other words,
			\begin{equation*}
				\sup_{\zeta \in D_{\Theta}} \left| \int_{\dD} \phi(z) d\nu_{\lambda}(z) \right| \leq \norm{A_{\phi}},
			\end{equation*}
			where
			\begin{equation*}
				d\nu_{\lambda}(z) :=  \frac{1}{|\Theta'(\zeta)|}
				\left|\frac{\Theta(z) - \Theta(\zeta)}{z - \zeta}\right|^2 \frac{|dz|}{2 \pi}
			\end{equation*}
			is a family of probability measures on $\dD$ indexed by $\zeta \in D_{\Theta}$.
	\end{Corollary}

\section{Lower bounds and projections}

	Our next several results concern lower bounds on $\norm{A_{\phi}}$
	involving the orthogonal projection $P_{\Theta}: L^2 \to K_{\Theta}$.

	\begin{Theorem}\label{TheoremLowerBound}
		If $\Theta$ is an inner function and $\phi \in L^2$, then
		\begin{equation*}
			\frac{  \norm{ P_{\Theta} (\phi) - \overline{\Theta(0)} P_{\Theta} (\Theta \phi)}}
			 { ( 1 - |\Theta(0)|^2 )^{\frac{1}{2}} }  \, \leq \,
			 \norm{A_{\phi}}.
		\end{equation*}
	\end{Theorem}

	\begin{proof}
		First observe that $\norm{k_0} = ( 1 - |\Theta(0)|^2  )^{\frac{1}{2}}$.
		Next we see that if $\phi \in L^2$ and $g \in K_{\Theta}$ is any unit vector, then
		\begin{align*}
			( 1 - |\Theta(0)|^2 )^{\frac{1}{2}} \norm{A_{\phi}}
			&\geq  | \inner{ A_{\phi} k_0, g} |  \\
			&=  | \inner{ P_{\Theta} (\phi k_0),g} | \\
			&=  | \inner{ P_{\Theta} (\phi) - \overline{\Theta(0)} P_{\Theta}(\Theta \phi),g} |.
		\end{align*}
		Setting
		\begin{equation*}
			g = \frac{P_{\Theta} (\phi) - \overline{\Theta(0)} P_{\Theta} (\Theta \phi)}
			{\norm{P_{\Theta} (\phi) - \overline{\Theta(0)} P_{\Theta} (\Theta \phi)}}
		\end{equation*}
		yields the desired inequality.
	\end{proof}
	
	In light of the fact that $P_{\Theta}(\Theta \phi) = 0$ whenever $\phi \in H^2$,
	Theorem \ref{TheoremLowerBound} leads us immediately to the following corollary:

	\begin{Corollary} \label{Cor-main1}
		If $\Theta$ is inner and $\phi \in H^2$, then
		\begin{equation} \label{eq-Cor-main1}
			\frac{\norm{P_{\Theta} (\phi)}}{(1 - |\Theta(0)|^2)^{1/2}} \leq \norm{A_{\phi}}.
		\end{equation}
	\end{Corollary}

	It turns out that \eqref{eq-Cor-main1} has a rather interesting function-theoretic
	implication.  Let us first note that for $\phi \in H^{\infty}$,  we can expect no better inequality than
	\begin{equation*}
		\norm{\phi} \leq \norm{\phi}_{\infty}
	\end{equation*}
	(with equality holding if and only if $\phi$ is a scalar multiple of an inner function).  However, if
	$\phi$ belongs to $K_{\Theta}^{\infty}$, then a stronger inequality holds.

	\begin{Corollary}\label{CorollaryOddRelationship}
		If $\Theta$ is an inner function, then
		 \begin{equation}\label{eq-ModelFunctionInequality}
			 \norm{\phi} \,\leq\, ( 1 - |\Theta(0)|^2 )^{\frac{1}{2}} \norm{\phi }_{\infty}
		\end{equation}
		holds for all $\phi \in K_{\Theta}^{\infty}$.
		If $\Theta$ is a finite Blaschke product, then equality holds
		if and only if
		$\phi$ is a scalar multiple of an inner function from $K_{\Theta}$.
	\end{Corollary}
	
	\begin{proof}
		First observe that the inequality
		\begin{equation*}
			\norm{ \phi  } \leq ( 1 - |\Theta(0)|^2 )^{\frac{1}{2}} \norm{\phi }_{\infty}
		\end{equation*}
		follows from Corollary \ref{Cor-main1} and the fact that $P_{\Theta}\phi = \phi$
		whenever $\phi \in K_{\Theta}$.  Now suppose that $\Theta$ is a finite Blaschke product
		and assume that equality
		holds in the preceding for some $\phi \in K_{\Theta}^{\infty}$.  In light of \eqref{eq-Cor-main1},
		it follows that $\norm{A_{\phi}} = \norm{\phi}_{\infty}$.   From \cite[Thm.~2]{NLEPHS}
		we see that $\phi$ must be a scalar multiple of the inner \textit{part} of a function
		from $K_{\Theta}$. But since $\phi \in K_{\Theta}^{\infty}$, then $\phi$ must be a scalar multiple of an inner function from $K_{\Theta}$.
%Regarded as a space of functions on $\dD$, one can show \cite{CR} that
%		\begin{equation}\label{eq-BoundaryDecomp}
%			K_{\Theta} = H^2 \cap \Theta \overline{z H^2}
%		\end{equation}
%		whence, since $\Theta(0) \not = 0$, a short argument shows that $\phi$
%		can not be a scalar multiple of an inner function.
	\end{proof}
	
	When $\Theta$ is a finite Blaschke product, then $K_{\Theta}$ is a
	finite dimensional subspace of $H^2$ consisting of bounded
	functions \cite{MR2440673, G,Sarason}. By elementary functional analysis, there are $c_1, c_2 > 0$ so that
	\begin{equation*}
		c_1 \norm{\phi} \leq \norm{\phi}_{\infty} \leq c_2 \norm{\phi}
	\end{equation*}
	for all $\phi \in K_{\Theta}$. This prompts the following question:

	\begin{Question}
		What are the optimal constants $c_1, c_2$ in the above inequality?
	\end{Question}

\section{Lower bounds from the decomposition of $K_{\Theta}$}

	A result of Sarason \cite[Thm.~3.1]{Sarason} says, for $\phi \in L^2$, that
	\begin{equation} \label{zero-operator}
		A_{\phi} \equiv 0 \Leftrightarrow \phi \in \Theta H^2 + \overline{\Theta H^2}.
	\end{equation}
	 It follows that
	the most general truncated Toeplitz operator on $K_{\Theta}$ is of the form
	$A_{\psi+ \overline{\chi}}$ where $\psi, \chi \in K_{\Theta}$.
	%Sarason also notes that $A_{k_0} = I$ and hence the symbol
%	\begin{equation*}
%		(\psi + \overline{\chi(0)}k_0) + \overline{ (\chi - \chi(0)k_0) }
%	\end{equation*}
%	yields the same truncated Toeplitz opeator as $\psi+ \overline{\chi}$ \cite[p.~499]{Sarason}.
%	Therefore each bounded truncated Toeplitz operator on $K_{\Theta}$ is induced by a unique
%	symbol of the form $\phi = \psi + \overline{\chi}$ where $\phi,\chi \in K_{\Theta}$
%	and $\chi(0) = 0$.
We can refine this observation a bit further and provide another canonical decomposition
	for the symbol of a truncated Toeplitz  operator.

	\begin{Lemma}\label{LemmaSymbol}
		Each bounded truncated Toeplitz operator on $K_{\Theta}$ is generated by a symbol of the form
		\begin{equation}\label{eq-SpecialForm}
			\phi = \underbrace{ \psi }_{\in H^2} + \underbrace{ \chi \overline{\Theta} }_{\in \overline{zH^2}}
		\end{equation}
		where $\psi, \chi \in K_{\Theta}$.
	\end{Lemma}

	Before getting to the proof, we should remind the reader of a technical detail.
	It follows from the identity $K_{\Theta} = H^2 \cap \Theta \overline{z H^2}$ (see \cite[p.~82]{CR}) that
	\begin{equation*}
	C: K_{\Theta} \to K_{\Theta}, \quad C f := \overline{z f} \Theta,
	\end{equation*}
	is an isometric, conjugate-linear, involution. In fact,
	when $A_{\phi}$ is a bounded operator we have the identity $C A_{\phi} C = A_{\phi}^{*}$ \cite[Lemma 2.1]{Sarason}.

	\begin{proof}[Proof of Lemma \ref{LemmaSymbol}]
		If $T$ is a bounded truncated Toeplitz operator on $K_{\Theta}$, then there exists some
		$\phi \in L^2$ such that $T = A_{\phi}$.  We claim that this $\phi$ can be chosen to have the special form
		\eqref{eq-SpecialForm}.  First let us write $\phi = f + \overline{zg}$ where $f,g \in H^2$.
		Using the orthogonal decomposition $H^2 = K_{\Theta} \oplus \Theta H^2$, it follows that $\phi$
		may be further decomposed as
		\begin{equation*}
			\phi = (f_1 + \Theta f_2) + \overline{z(g_1 + \Theta g_2)}
		\end{equation*}
		where $f_1, g_1 \in K_{\Theta}$ and $f_2,g_2 \in H^2$.
		By \eqref{zero-operator}, the symbols $\Theta f_2$ and
		$\overline{ \Theta(zg_2) }$ yield the zero truncated Toeplitz operator on $K_{\Theta}$.
		Therefore we may assume that
		\begin{equation*}
			\phi = f + \overline{zg}
		\end{equation*}
		for some $f,g \in K_{\Theta}$.
		Since $C g = \overline{gz}\Theta$, we have
		$\overline{zg} = (C g) \overline{\Theta}$
		and hence \eqref{eq-SpecialForm} holds with
		$\psi = f$ and $\chi = C g$.
	\end{proof}

	\begin{Corollary}\label{PropMixed}
		Let $\Theta$ be an inner function.  If $\psi_1, \psi_2 \in K_{\Theta}$
		and $\phi = \psi_1 + \psi_2\overline{\Theta}$, then
		\begin{equation*}
			\frac{ \| \psi_1 - \overline{ \Theta(0) } \psi_2\|}{ (1 - |\Theta(0)|^2)^{\frac{1}{2}}} \leq \norm{A_{\phi}}.
		\end{equation*}
	\end{Corollary}

	\begin{proof}
		If $\phi = \psi_1 + \psi_2\overline{\Theta}$, then, since $\psi_1, \psi_2 \in K_{\Theta}$
		and $\psi_{2} \overline{\Theta} \in \overline{z H^2}$, we have
		\begin{equation*}
			P_{\Theta} (\phi) - \overline{\Theta(0)} P_{\Theta} (\Theta \phi) = \psi_1 - \overline{\Theta(0)} \psi_2.
		\end{equation*}
		The result now follows from Theorem \ref{TheoremLowerBound}.
	\end{proof}

\bibliography{NTTO}

\end{document}